\documentclass[a4, 12pt]{article}
\usepackage[utf8]{inputenc}
\usepackage{t1enc}
\usepackage{amsfonts,amsmath,amssymb,amsthm,color,amsbsy}
\usepackage{color}
\usepackage{url}
\usepackage{eucal}
\usepackage{enumerate}
\usepackage{graphicx}
\usepackage{mathrsfs}
\usepackage{bm} 
\usepackage{enumerate}
\usepackage{amssymb}
\usepackage{amsmath}
\usepackage{amsthm}
\usepackage{amsfonts}
\usepackage{bbm}
\usepackage{anysize}
\marginsize{2 cm}{2 cm}{2 cm}{2 cm}
\theoremstyle{definition}
\newtheorem{df}{Definition}
\newtheorem{uw}[df]{Remark}
\theoremstyle{plain}

\newtheorem{thm}[df]{Theorem}
\newtheorem{prop}[df]{Proposition}
\newtheorem{lm}[df]{Lemma}
\newtheorem{wn}[df]{Corollary}

\def\Ker{\mathrm{Ker}\,}
\begin{document}
\title{Period Estimates for Autonomous Evolution Equations with Lipschitz Nonlinearities}
\author{Aleksander \'Cwiszewski, W{\l}adys{\l}aw Klinikowski,\\
{\em Faculty of Mathematics and Computer Science},\\
{\em Nicolaus Copernicus University,  Toru\'n, Poland}\\
{\em e-mail: aleks@mat.umk.pl, wklin@mat.umk.pl}}
\maketitle
\begin{abstract}
We derive an estimate for the minimal period of autonomous strongly damped hyperbolic problems. Our result corresponds to the works by Yorke \cite{Yorke}, Busenberg et al. \cite{Busenberg-et-al} for ordinary differential equations as well as Robinson and Vidal-Lopez \cite{Rob-Vid-Lop-2006} and \cite{Rob-Vid-Lop} for parabolic problems. A general approach is developed for treating both hyperbolic and parabolic problems. An example of application to a class of beam equations is provided. 
\end{abstract}
\section{Introduction}
The known estimates for ordinary differential equations comes from Yorke who proved in \cite{Yorke} that if $T>0$ is the minimal period of a solution to  an ordinary differential equation
\begin{equation}\label{ode}
\dot u(t) = f(u(t)),
\end{equation} 
where $f:\mathbb{R}^N\to \mathbb{R}^N$ is Lipschitz with constant $L>0$, then $T \geq 2\pi/L$. 
The same estimate was showed in the infinite dimensional phase space by Busenberg et al. in \cite{Busenberg-et-al} where \eqref{ode} is considered in a Hilbert space. For \eqref{ode} in a general Banach space, the estimate $T \geq 6/L$ was proved therein. A natural question arises if there is any period estimate in case of partial differential equations and systems.  In case of the abstract parabolic problems of the form 
\begin{equation}\label{parabolic-ev-eq}
\dot u(t) + A u(t) =  f(u(t)),\, t>0,
\end{equation}
where $A$ is a self-adjoint positive operator in a separable Hilbert space $X$ and a Lipschitz function $f:X^{\beta}\to X$ is defined on  the fractional space $X^\beta$ with $\beta\in [0,1)$  (associated to the operator $A$ -- see e.g. \cite{Henry}), Robinson and Vidal-Lopez in  \cite{Rob-Vid-Lop} (see also \cite{Rob-Vid-Lop-2006}) obtained the following lower bound 
\begin{equation}\label{Robinson-Vidal-Lopez-estimate}
T \geq \left(2^{1-\beta}+(\beta/e)^{\beta}/(1-e^{-1/2})(1-\beta) \right)^{-1/(1-\beta)}\cdot L^{- 1 / (1 - \beta)}.
\end{equation}
In this paper we shall revisit the parabolic problem, slightly improving the above estimate \eqref{Robinson-Vidal-Lopez-estimate} and deal with the minimal period for hyperbolic problems of the form 
\begin{equation}\label{E25}
\ddot{u}(t)+\alpha A \dot u(t) + Au(t) = f(u(t), \dot u(t)),\ t > 0,
\end{equation}
where $A$ is as above and $\alpha>0$ is a damping coefficient. In the hyperbolic case we assume, additionally, that $A$ has compact resolvent, which means that the spectrum $\sigma(A)$ consists of positive $\lambda_n$, $n\in \mathbb{N}$, such that $\lambda_n \to +\infty$ as $n\to +\infty$.
The nonlinear term $f:X^{1/2}\times X \to X$ ($X^{1/2}$ is the fractional space determined by $A$) is Lipschitz, i.e. there exists $L>0$ such that, for any $u_1, u_2\in X^{1/2}$ and $v_1, v_2\in X$,
\begin{equation}\label{E101}
\|f(u_1,v_1)-f(u_2,v_2)\|\leq L\left(\|u_1-u_2\|_{1/2}^{2}+\|v_1-v_2\|^{2}\right)^{1/2}
\end{equation}
where $\|\cdot\|_{1/2}$ stands for the norm in $X^{1/2}$. It is an abstract model for many physical (systems of) equations, including systems with the so-called strongly damped beam that is fixed at both ends (see Section 6).\\
\indent Let us make general comments on the hyperbolic case. When $\alpha=0$ there is no lower bound for periods of periodic solutions of \eqref{E25}. Indeed, if $e\in X\setminus \{0\}$ is the eigenvector for $A$, corresponding to an eigenvalue $\lambda>0$ and $f(u,v):=L\cdot u$, with $L<\lambda$, then $u(t):=\sin(t\sqrt{\lambda-L}) \cdot e$ is a well-defined periodic solutions of \eqref{E25} with the minimal period $2\pi/\sqrt{\lambda-L}$. Therefore, if the eigenvalues of the operator $A$ make an unbounded set, then we have a periodic solution of \eqref{E25} of arbitrarily small period. Another case where one can not expect any minimal period estimate is when \eqref{E25} is gradient-like, that is there exists a Lyapunov functional, i.e. a functional that decreases/increases along nontrivial trajectories (see e.g. \cite{Hale-book}). Then obviously \eqref{E25} has no nonstationary periodic solutions. However, like in the parabolic case, there are classes of nonlinearities $f$ for which hyperbolic partial differential equations with strong damping are not gradient-like (see Section 6) and when periodic solutions occur. Note that also systems of hyperbolic equations are not gradient-like if only the linear perturbation field $f$ does not come from a gradient field.\\
\indent In order to study \eqref{E25} we rewrite the problem as a system
$$ \left\{  \begin{array}{l}
\dot u = v \\
\dot v = -A (\alpha \cdot v+u) + f (u,v),     
\end{array} \right. $$
which in turn can be represented as the first order equation
\begin{equation}\label{1st-order}
\dot z = {\bold A} z + {\bold F} (z),\, t>0,
\end{equation}
where $z=(u,v)$ and the operator 
${\bold A}: D({\bold A}) \to {\bold X}$ in ${\bold X} := X^{1/2} \times X$ is given by
\begin{equation}\label{boldA-formula}
{\bold A} (u,v) :=  (v, -A (\alpha \cdot v+u) ), \ (u,v) \in D({\bold A}),
\end{equation}
with $D({\bold A}) := \{ (u,v) \in {\bold X}\, \mid \, \alpha \cdot v+u \in D(A),\, v\in X^{1/2} \}$.
The nonlinear term ${\bold F}:{\bold X}\to {\bold X}$ is given by ${\bold F}(u,v)=(0, f(u,v))$ and is also Lipschitz with the constant $L$ if ${\bold X}$ is equipped with the norm given by the scalar product
$\langle(u_1, v_1),(u_2,v_2)\rangle = \langle u_1, u_2{\rangle}_{1/2} + \langle v_1, v_2\rangle$. It is well-known that $-{\bold A}$ is sectorial (see e.g. \cite{Fitzgibbon}, \cite{Massatt} or \cite{Cwiszewski_Rybakowski}). One cannot apply the result by Robinson and Vidal-Lopez from \cite{Rob-Vid-Lop} due to the fact that ${\bold A}$  is not symmetric (as it is assumed therein), neither a direct application of the idea sketched in remarks after the proof of Theorem 3.2 in \cite{Rob-Vid-Lop} will provide satisfactory estimates. In this paper we develop an approach that can be effectively applied to the hyperbolic case.\\ 
\indent We come up with a unified approach allowing to consider both parabolic and hyperbolic  partial differential systems/equations. We first consider a general operator $\mathbb{A}: D(\mathbb{A})\subset\mathbb{X}\to \mathbb{X}$
being an infinitesimal generator of a $C_0$ semigroup of bounded linear operators $e^{t\mathbb{A}}:\mathbb{X}\to \mathbb{X}$, $t\geq 0$, on a Banach space $(\mathbb{X}, \|\cdot\|_{\mathbb{X}} )$. Assume that $( \mathbb{V} , \| \cdot \|_{\mathbb{V}} )$ is a Banach space such that\\
$(V_1)$ \parbox[t]{144mm}{$D(\mathbb{A})\subset \mathbb{V} \subset \mathbb{X}$;}\\[0.5em]
$(V_2)$ \parbox[t]{144mm}{$\mathbb{V}$ is embedded continuously into $\mathbb{X}$;  
}\\[0.5em]
$(V_3)$ \parbox[t]{144mm}{$e^{t\mathbb{A}} (\mathbb{X})\subset \mathbb{V}$, for all $t>0$, and the family $\{ e^{t\mathbb{A}} |_\mathbb{V}:\mathbb{V}\to \mathbb{V}\}_{t\geq 0}$, of operators restricted to $\mathbb{V}$, is a $C_0$ semigroup of bounded linear operators on $\mathbb{V}$.}\\[0.5em]
We say that $\mathbb{A}$ has the {\em uniformly half-bounded decomposition} property, in short property $(UHBD)$, whenever there exists $\mu_0\geq 0$, $M>0$ and a continuous $m: (0,+\infty) \to (0,+\infty)$ such that, for any $\mu>\mu_0$, there exists a decomposition $\mathbb{X} = \mathbb{X}_{\mu}^{-} \oplus \mathbb{X}_{\mu}^{+}$ into closed subspaces such that\\[0.5em]
$(D_{\mu,+})$ \parbox[t]{144mm}{ $\mathbb{X}_{\mu}^{+}\subset D(\mathbb{A})$, $\mathbb{A}\mathbb{X}_{\mu}^{+} \subset \mathbb{X}_{\mu}^{+}$ and $\|\mathbb{A} w \|_{\mathbb{V}} \leq \mu M \|w\|_{\mathbb{V}}$,   for all $w\in \mathbb{X}_{\mu}^{+}$;}\\[0.5em]
$(D_{\mu,-})$ \parbox[t]{144mm}{
$\mathbb{A}(\mathbb{X}_{\mu}^{-}\cap D(\mathbb{A})) \subset 
\mathbb{X}_{\mu}^{-}$, $\|e^{t\mathbb{A}} w\|_{\mathbb{X}}\leq e^{-\mu t}\|w\|_{\mathbb{X}}$,  $\|e^{t\mathbb{A}} w\|_{\mathbb{V}}\leq m(t) e^{-\mu t}\|w\|_{\mathbb{X}}$, for all $w\in \mathbb{X}_{\mu}^{-}$ and $t>0$, and $\|e^{t\mathbb{A}} w\|_{\mathbb{V}}\leq e^{-\mu t}\|w\|_{\mathbb{V}}$,  for all  $w\in \mathbb{X}_{\mu}^{-}\cap \mathbb{V}$ and $t>0$;
}\\[0.5em]
$(D_{\mu,0})$ \parbox[t]{144mm}{$\int_{0}^{t} m(s)e^{-\mu s}\, ds < +\infty$ for all $t>0$.
}\\[0.5em]
A continuous function $z: (T_1,T_2) \to \mathbb{V}$, $T_1<T_2$, is said to be  {\em a mild solution} of
\begin{equation}\label{abstract-ev-eq}
\dot{z} = \mathbb{A}z+\mathbb{F} (z) 
\end{equation}
if and only if, for any $t,t_0 \in (T_1,T_2)$ with $t_0<t$,
\begin{equation}\label{Duhamel-eq-def}
z(t) = e^{(t-t_0)\mathbb{A}} z(t_0) + \int_{t_0}^{t} e^{(t-s)\mathbb{A}} \mathbb{F}(z(s))\, ds. 
\end{equation}
In this abstract setting we get the following result.
\begin{thm} \label{main-res-abs}
Suppose that the generator $\mathbb{A}$ of a $C_0$-semigroup of bounded linear operators in a Banach space $\mathbb{X}$ satisfies $(V_1)$--$(V_3)$ and has property $(UHBD)$ with some Banach space $\mathbb{V}\subset \mathbb{X}$ and $\mathbb{F}:\mathbb{V}\to \mathbb{X}$ is Lipschitz with constant $L>0$. If there exists a nonstationary $T$-periodic mild solution $z:\mathbb{R}\to \mathbb{V}$ of 
\eqref{abstract-ev-eq}, then either $T\geq 1/\mu_0 M$ or, for all $\mu\in(\mu_0, 1/MT)$, 
\begin{equation}\label{technical-ineq}
1 \leq   T L 
\cdot \left[ \frac{K_{\mu}^{+}}{1-\mu M T} +  
\frac{K_{\mu}^{-}}{\mu T} \left(e^{-\mu T}m(T)  + \int_{0}^{\mu T} m(s/\mu)\cdot e^{- s} \, ds \right) \right]
\end{equation}
with $K_\mu^{+}:= \| \mathbb{P}_{\mu}^{+}\|_{{\cal L}(\mathbbm{X},\mathbbm{V})}$ and $K_\mu^{-}:= \| \mathbb{P}_{\mu}^{-}\|_{{\cal L}(\mathbbm{X},\mathbbm{X})}$, 
where $\mathbb{P}_{\mu}^{+}: \mathbbm{X} \to \mathbb{X}_{\mu}^{+}$
and $\mathbb{P}_{\mu}^{-}: \mathbbm{X} \to \mathbb{X}_{\mu}^{-}$ are the projections.
\end{thm}
\noindent The above estimate does not provide an explicit estimate for $T$,  however in a special case we get the following result, which applies in the case of damped hyperbolic problems \eqref{E25} (see Theorem \ref{main-res} below).
\begin{wn}\label{cor-1}
If additionally to the assumptions of Theorem \ref{main-res-abs} we assume that $\mu_0>0$, $m\equiv 1$ and
\begin{equation}\label{additional-as-main-abs-res}
K_{\mu}^{-} \leq (1-\mu_0/\mu)^{-1}, \ \   K_{\mu}^{+} \leq (1-\mu_0/\mu)^{-1},\ \  \mbox{ for all } \ \mu>\mu_0,
\end{equation}
then the period of any nontrivial periodic solution of \eqref{abstract-ev-eq} satisfies the inequality
$$
T\geq  1/L\left(1+\sqrt{M(1+\mu_0/L)}\right)^2.
$$ 
\end{wn}
\noindent The proof of Theorem \ref{main-res-abs} and Corollary \ref{cor-1} is presented in Section 2.\\
\indent In Section 3 we shall apply the above setting to the parabolic evolution problem \eqref{parabolic-ev-eq} and provide an estimate for the minimal period of parabolic systems.
\begin{thm} \label{main-res-par}
If $u:[0,T]\to X^{\beta}$, $\beta\in [0,1)$ is a nonstationary periodic solution of \eqref{parabolic-ev-eq} with the period $T>0$, then
\begin{equation}\label{main-res-est}
T \geq  1/(L K_\beta)^{1/(1-\beta)}
\end{equation}
where $K_\beta := \min_{\eta\in(0,1)} H(\eta)$ and $H:(0,1)\to (0,+\infty) $ is given by
\begin{equation}\label{par-H-eta-def}
H(\eta):= \frac{\eta^\beta}{1-\eta} + \frac{M_\beta}{\eta} \left(
e^{-\eta} + \eta^\beta \int_{0}^{\eta} s^{-\beta} e^{-s} \, ds 
\right)
\end{equation}
with $M_\beta= (\beta/e)^{\beta}$ if $\beta\in (0,1)$ and $M_\beta = 1$ if $\beta = 0$.\\
\indent Consequently, one has
\begin{equation}\label{1/2-better-estimate}
T\geq 1/ L^{1/(1-\beta)} \cdot \left(2^{1-\beta} + 2 M_\beta/e^{1/2} + M_\beta/(1-\beta)  \right)^{1/(1-\beta)} \ \ \mbox{ for } \beta\in (0,1),
\end{equation}
and, $$
T\geq 1/4 L   \ \ \mbox{ for } \beta =0. 
$$
\end{thm}
\noindent The above result improves \eqref{Robinson-Vidal-Lopez-estimate} obtained by Robinson and Vidal-Lopez in \cite{Rob-Vid-Lop} (see Remark \ref{Rob-Vid-Lop-comp}).\\
\indent In Section 4 we study the spectral properties of the hyperbolic operator ${\bold A}$ given by \eqref{boldA-formula} and in Section 5 we prove that ${\bold A}$ has the property $(UHBD)$ with constants $\mu_0=2/\alpha$ and $M=1+\sqrt{2}$ (see Corollary \ref{to-obtain-(D)-property}) as well as the inequalities \eqref{additional-as-main-abs-res} hold (see Corollary \ref{to-obtain-(D)-property-1}). Therefore, in view of Corollary  \ref{cor-1}, we get the following minimal period estimate for damped hyperbolic equations, which is the main result of the paper.
\begin{thm} \label{main-res}
If $u:[0,T]\to X$ is a nonstationary periodic solution of \eqref{E25} with the period $T>0$, then
\begin{equation}\label{main-res-est-bis}
T \geq 1/L\left(1+\sqrt{(1+1/\sqrt{2})(1+2/\alpha L)}\right)^2.
\end{equation}
\end{thm}
\noindent We shall illustrate the result in case of a damped beam equation in Section 6. 

\section{General period estimate -- proof of Theorem \ref{main-res-abs}}
Assume that the operator $\mathbb{A}$ and $\mathbb{F}:\mathbb{V}\to \mathbb{X}$ is as in Theorem \ref{main-res-abs}. Suppose that $z:\mathbb{R} \to \mathbb{V}$ is a mild solution of 
$$
\dot z = \mathbb{A}z+\mathbb{F}(z) \  \mbox{ on } \ \mathbb{R}
$$
with minimal period $T>0$. In particular, $z(t)=z(t+T)$, for all $t\in\mathbb{R}$, and one has the Duhamel formula
\begin{equation}\label{abstract-Duhamel}
z(t)=e^{ (t-t_0) \mathbb{A} } z(t_0) + \int_{t_0}^{t} e^{(t-s)\mathbb{A}} \mathbb{F} ( z(s) ) \, d s
\end{equation}
for any $t_0\in\mathbb{R}$ and all $t>t_0$. This yields, for any $t\in\mathbb{R}$,
$$
z(t)=z(t+T)=e^{T\mathbb{A}}z(t)+\int_{t}^{t+T}e^{(t+T-s)\mathbb{A}}\mathbb{F}(z(s))\, ds,
$$
which, after change of variables in the integral, yields
\begin{equation}\label{P_plus-eq}
(I-e^{T\mathbb{A}})z(t)=\int_{0}^{T} e^{(T-s)\mathbb{A}}\mathbb{F}(z(t+s))\, ds.
\end{equation}
Let us take $\mu>\mu_0$ such that
\begin{equation}
\mu M T < 1    
\end{equation}
and use property $(UHBD)$ to obtain the decomposition
$$
\mathbb{X} = \mathbb{X}_{\mu}^{-} \oplus \mathbb{X}_{\mu}^{+}.
$$
\begin{uw}  \label{D-remark}
By $(D_{\mu,+})$ and $(D_{\mu,-})$, one has $\mathbb{A}\mathbb{P}_{\mu}^{+}w = \mathbb{P}_{\mu}^{+} \mathbb{A}w$ for all $w\in D(\mathbb{A})$ and $\mathbb{A}\mathbb{P}_{\mu}^{-}w = \mathbb{P}_{\mu}^{-} \mathbb{A}w$ for all $w\in D(\mathbb{A})\cap {\mathbb{X}}_{\mu}^{-}$.
\end{uw}
We shall also need the following estimates.
\begin{lm}\label{V-trick-lemma}
Under the above assumptions, for any $\mu>\mu_0$ and $T>0$,  define $R_{\mu,T}: \mathbb{X}_{\mu}^{-}\cap \mathbb{V} \to \mathbb{X}_{\mu}^{-}\cap \mathbb{V}$ by
$R_{\mu,T} u:=u-e^{T \mathbb{A}}u$, $u\in \mathbb{X}_{\mu}^{-}\cap \mathbb{V}$.
Then
$$
\| R_{\mu,T}^{-1} w\|_{\mathbb{V}} \leq (1-e^{-\mu T})^{-1} \| w \|_{\mathbb{V}} 
\mbox{ for all } w\in \mathbb{X}_{\mu}^{-} \cap \mathbb{V}, \leqno{(i)}
$$
$$
\|R_{\mu,T}^{-1} w - w \|_{\mathbb{V}} \leq e^{-\mu T} (1-e^{-\mu T})^{-1} m(T) \|w\|_{\mathbb{X}} \mbox{ for all } w\in \mathbb{X}_{\mu}^{-}\cap\mathbb{V}. \leqno{(ii)}
$$
\end{lm}
\begin{proof} 
(i) By $(D_{\mu,-})$  we have $e^{T\mathbb{A}} \mathbb{X}_{\mu}^{-} \subset \mathbb{X}_{\mu}^{-}$
and 
\begin{equation}\label{semigroup-norm<1}
\|e^{T\mathbb{A}}\|_{{\cal L}(\mathbb{X}_{\mu}^{-} \cap \mathbb{V},\mathbb{X}_{\mu}^{-}\cap \mathbb{V})}\leq e^{-\mu T}<1, 
\end{equation}
where in the closed subspace 
$\mathbb{X}_{\mu}^{-}\cap \mathbb{V}$ of $\mathbb{V}$ we consider the norm from $\mathbb{V}$. Hence, we can infer that the operator $R_{\mu, T}$ is invertible and 
\begin{equation}\label{plus_semigroup_estimate}
\|R_{\mu,T}^{-1}\|_{{\cal L}(\mathbb{X}_{\mu}^{-}\cap \mathbb{V},\mathbb{X}_{\mu}^{-}\cap \mathbb{V})} \leq (1-\|e^{T\mathbb{A}}\|_{{\cal L}(\mathbb{X}_{\mu}^{-}\cap \mathbb{V},\mathbb{X}_{\mu}^{-}\cap \mathbb{V})})^{-1}
\leq  (1-e^{-\mu T})^{-1}.
\end{equation}
\indent (ii) According to (i) and $(D_{\mu,-})$, for any $w\in \mathbb{X}_{\mu}^{-} \cap \mathbb{V}$, we get
\begin{eqnarray*}
\| R_{\mu,T}^{-1} w - w\|_{\mathbb{V}} & = & \left\|\sum_{k=1}^{\infty} e^{kT\mathbb{A}}w \right\|_{\mathbb{V}}  =   \left\| \sum_{k=0} e^{kT\mathbb{A}} e^{T\mathbb{A}}  w \right\|_{\mathbb{V}}\\
& = & \left\| R_{\mu,T}^{-1} e^{T\mathbb{A}}  w\right\|_{\mathbb{V}} \leq 
 (1-e^{-\mu T})^{-1} \left\|e^{T\mathbb{A}}  w\right\|_{\mathbb{V}}\\ 
& \leq & (1-e^{-\mu T})^{-1} e^{-\mu T} m(T)\|w\|_{\mathbb{X}},
\end{eqnarray*}
which completes the proof.
\end{proof}

Now choose $\tau\in (0,T)$ and consider $D:[0,+\infty)\to \mathbb{X}$ given by 
$$
D(t):=z(t+\tau)-z(t), \ t\geq 0.
$$
Clearly, $D$ is a nonzero function as $T$ is the minimal period of $z$ and we have the following estimates.
\begin{lm}\label{D_fun-lemma}
Under the above assumptions if $\mu_0 < 1/MT$, then, for any $\mu \in (\mu_0, 1/MT)$, 
$$
\|\mathbb{P}_{\mu}^{+}D\|_{L^\infty (0,T;\mathbb{V})} \leq 
T L K_{\mu}^{+} (1-\mu M T)^{-1} \|D\|_{L^\infty (0,T;\mathbb{V})}, \leqno{(i)}
$$
$$
\|\mathbb{P}_{\mu}^{-}D\|_{L^\infty (0,T;\mathbb{V})}\leq
 K_{\mu}^{-} L \left(\mu^{-1}e^{-\mu T}m(T)  + \int_{0}^{T} m(s)\cdot e^{-\mu s} \, ds \right) \|D\|_{L^\infty (0,T;\mathbb{V})}. \leqno{(ii)}
$$
\end{lm}
\begin{proof} Observe that, in view of \eqref{abstract-Duhamel},
\begin{align*}
D(t) = z(t+\tau)-z(t)= e^{t\mathbb{A}}D(0)+\int_{0}^{t}e^{(t-s)\mathbb{A}} [\mathbb{F}(z(s+\tau))-\mathbb{F}(z(s))]\, ds  \ \ \mbox{ for all } t\geq 0. 
\end{align*}
Acting with $\mathbb{P}_{\mu}^{+}$ on both sides, one gets
$$
\mathbb{P}_{\mu}^{+}D(t) = e^{t\mathbb{A}}\mathbb{P}_{\mu}^{+}D(0) 
+ \int_{0}^{t}e^{(t-s)\mathbb{A}} \mathbb{P}_{\mu}^{+}w(s)\, ds
$$
with $w(s):=\mathbb{F}(z(s+\tau))-\mathbb{F}(z(s))$. Since $\mathbb{A}$ is bounded on $\mathbb{X}_{\mu}^{+}$ we get
\begin{equation}\label{der-p-min-D-1}
(\mathbb{P}_{\mu}^{+}D)'(t)= \mathbb{A} \mathbb{P}_{\mu}^{+}D(t)+ \mathbb{P}_{\mu}^{+}w(t), \ \mbox{ for any } t\geq 0.
\end{equation}
On the other hand
$$
\int_{0}^{T} \mathbb{P}_{\mu}^{+}D(s) \, ds =  \mathbb{P}_{\mu}^{+} \left( 
\int_{0}^{T} z(s+\tau) \, ds - \int_{0}^{T} z(s) \, ds\right) = 0,
$$
which implies
\begin{align*}
\mathbb{P}_{\mu}^{+} D(t) & =  \frac{1}{T}
\int_{0}^{T} \mathbb{P}_{\mu}^{+} D(t)\, dr  =
 \frac{1}{T}\int_{0}^{T} \left( 
\mathbb{P}_{\mu}^{+} D(r) + \int_{r}^{t} (\mathbb{P}_{\mu}^{+} D)'(s)\,  ds\right)\, dr\\
 & = \frac{1}{T}\int_{0}^{T} \left( \int_{r}^{t} (\mathbb{P}_{\mu}^{+} D)'(s)\,  ds\right)\, dr.
\end{align*}
Observe that $\mathbb{X}_{\mu}^{+}\subset D(\mathbb{A})\subset \mathbb{V}$, therefore 
\begin{equation}\label{der-p-min-D-2}
\|\mathbb{P}_{\mu}^{+}D(t)\|_{\mathbb{V}} \leq \int_{0}^{t} \|(\mathbb{P}_{\mu}^{+} D)'(s)\|_{\mathbb{V}}\,  ds.
\end{equation}
\indent By use of \eqref{der-p-min-D-1}, \eqref{der-p-min-D-2} and $(D_{\mu,+})$, we obtain, for all $t\in [0,T]$,
\begin{align*}
\|\mathbb{P}_{\mu}^{+}D(t)\|_{\mathbb{V}}  & \leq  \int_{0}^{t} \|\mathbb{A}\mathbb{P}_{\mu}^{+}D(s) \|_{\mathbb{V}}\, ds + 
\int_{0}^{t} \|\mathbb{P}_{\mu}^{+}w(s) \|_{\mathbb{V}}\, ds\\
& \leq \mu  M  \int_{0}^{T} \|\mathbb{P}_{\mu}^{+} D(s) \|_{\mathbb{V}}\, ds + K_{\mu}^{+ }L\int_{0}^{T} \|D(s)\|_{\mathbb{V}}\, ds \\
& \leq \mu M T \|\mathbb{P}_{\mu}^{+}D\|_{L^\infty (0,T;\mathbb{V})} + T L 
K_{\mu}^{+}\|D\|_{L^\infty (0,T;\mathbb{V})} 
\end{align*}
and, in consequence,
$$
\|\mathbb{P}_{\mu}^{+}D\|_{L^\infty (0,T;\mathbb{V})} \leq T L K_{\mu}^{+} (1-\mu M T)^{-1}\|D\|_{L^\infty (0,T;\mathbb{V})},
$$
which proves (i).\\
\indent In order to show (ii), we use \eqref{P_plus-eq} and the invariance of $\mathbb{X}_{\mu}^{-}$ with respect to $e^{T\mathbb{A}}$ to get
$$
\mathbb{P}_{\mu}^{-}D(t) = (I-e^{T\mathbb{A}})^{-1} \int_{0}^{T} e^{(T-s)\mathbb{A}}\, \mathbb{P}_{\mu}^{-}w(t+s)\, ds.
$$
In view of $(D_{\mu,-})$ and $(D_{\mu,0})$, the integral 
$$
\int_{0}^{T} e^{(T-s)\mathbb{A}}\, \mathbb{P}_{\mu}^{-}w(t+s)\, ds
$$
which, by definition, is an element of the space $\mathbb{X}_{\mu}^{-}$, is also convergent in the space $\mathbb{V}$, since, by use of $(D_{\mu,-})$ and the Lipschitz property of $\mathbb{F}$, we get
\begin{align} \nonumber
\int_{0}^{T} \| e^{(T-s)\mathbb{A}} \mathbb{P}_{\mu}^{-}w(t+s)\|_{\mathbb{V}}\, ds
& \leq  \int_{0}^{T} m(T-s)e^{-\mu (T-s)} \|\mathbb{P}_{\mu}^{-} [\mathbb{F}(z(t+s+\tau))-\mathbb{F}(z(t+s))]\|_{\mathbb{X}}\,  ds \\
&  \leq  K_{\mu}^{-} L \int_{0}^{T} m(T-s) e^{-\mu(T-s)}\, \|D(t+s)\|_{\mathbb{V}}\, ds \nonumber \\
 & \leq K_{\mu}^{-}  L \left(\int_{0}^{T} m(s) e^{-\mu s} \, ds\right) \|D\|_{L^\infty (0,T;\mathbb{V})} < \infty. \nonumber 
\end{align}
Hence, by Lemma \ref{V-trick-lemma} (ii) and $(D_{\mu,-})$,
\begin{align} \nonumber
\| \mathbb{P}_{\mu}^{-}D(t)\|_{\mathbb{V}}
& \leq \left\|(R_{\mu,T}^{-1} - I) \int_{0}^{T} e^{(T-s)\mathbb{A}} \mathbb{P}_{\mu}^{-}w(t+s)\, ds \right\|_{\mathbb{V}} +\left\| \int_{0}^{T} e^{(T-s)\mathbb{A}} \mathbb{P}_{\mu}^{-}w(t+s)\, ds \right\|_{\mathbb{V}} \\
& \leq e^{-\mu T}(1-e^{-\mu T})^{-1} m(T) \left\| \int_{0}^{T} e^{(T-s)\mathbb{A}} \mathbb{P}_{\mu}^{-}w(t+s)\, ds \right\|_{\mathbb{X}} + \nonumber \\ 
&  \ \ \ \ +  K_{\mu}^{-}  L  \left(\int_{0}^{T} m(s) e^{-\mu s} \, ds\right)
\|D\|_{L^\infty (0,T;\mathbb{V})} \nonumber \\
& \leq e^{-\mu T}(1-e^{-\mu T})^{-1} m(T) K_\mu^{-} L 
\left(  \int_{0}^{T} e^{-\mu (T-s)} \, ds \right) \|D\|_{L^\infty (0,T;\mathbb{V})}  + \nonumber \\ 
&  \ \ \ \ +  K_{\mu}^{-}  L \left( \int_{0}^{T} m(s) e^{-\mu s} \, ds \right) 
\|D\|_{L^\infty (0,T;\mathbb{V})} \nonumber \\
& \leq K_{\mu}^{-} L \left(\mu^{-1}e^{-\mu T}m(T)  + \int_{0}^{T} m(s)\cdot e^{-\mu s} \, ds \right) \|D\|_{L^\infty (0,T;\mathbb{V})}, \nonumber  
\end{align}
which ends the proof of  (ii). \end{proof}

\begin{proof}[Proof of Theorem \ref{main-res-abs}] 
Since
$$
\|D\|_{L^\infty (0,T;\mathbb{V})} \leq \|\mathbb{P}_{\mu}^{+}D\|_{L^\infty (0,T;\mathbb{V})} + 
\|\mathbb{P}_{\mu}^{-}D\|_{L^\infty (0,T;\mathbb{V})}
$$
we see that, by use of Lemma \ref{D_fun-lemma},
$$
\|D\|_{L^\infty (0,T;\mathbb{V}) } \leq T L 
\cdot \left[ \frac{K_{\mu}^{+}}{1-\mu M T} +  
\frac{K_{\mu}^{-}}{\mu T} \left(e^{-\mu T}m(T)  + \int_{0}^{\mu T} m(s/\mu)\cdot e^{- s} \, ds \right) \right] \|D\|_{L^\infty (0,T;\mathbb{V})}
$$
Since $\|D\|_{L^{\infty}(0,T;\mathbb{X})}\neq 0$, we obtain the assertion \eqref{technical-ineq}. 
\end{proof}

\begin{uw}\label{alternative-procedure} 
In the proof of Lemma \ref{D_fun-lemma}, another possibility is to use \ref{V-trick-lemma} (i) to get
\begin{align} \nonumber
\|\mathbb{P}_{\mu}^{-}D(t)\|_{\mathbb{V}}
& \leq (1-e^{-\mu T})^{-1} \left\| \int_{0}^{T} e^{(T-s)\mathbb{A}} \mathbb{P}_{\mu}^{-}w(t+s)\, ds \right\|_{\mathbb{V}} \nonumber \\
& \leq (1-e^{-\mu T})^{-1} K_{\mu}^{-}  L \left(\int_{0}^{T} m(s)\cdot e^{-\mu s} \, ds\right) \|D\|_{L^\infty (0,T;\mathbb{V})}, \nonumber 
\end{align}
which is an alternative for the inequality (ii) in Lemma \ref{D_fun-lemma}. Following the proof of Theorem \ref{main-res-abs} we get the following version of its assertion:  if $z:\mathbb{R}\to \mathbb{V}$ is  a nonstationary $T$-periodic mild solution  of 
\eqref{abstract-ev-eq}, then either $T\geq 1/\mu_0 M$ or, for all $\mu\in(\mu_0, 1/MT)$, 
\begin{equation}\label{technical-ineq-alternative}
1 \leq   
T L  \cdot \left( \frac{K_{\mu}^{+}}{1-\mu M T} +   \frac{K_{\mu}^{-}}{ \mu  T(1-e^{-\mu T})} \int_{0}^{\mu T} m(s/\mu)\cdot e^{-s} ds\right).
\end{equation}
This means that, if only we knew that 
\begin{equation*}
\frac{1}{(1-e^{-\mu T})} \int_{0}^{\mu T} m(s/\mu)\cdot e^{-s} ds \geq e^{-\mu T}m(T)  + \int_{0}^{\mu T} m(s/\mu)\cdot e^{- s} \, ds,
\end{equation*}
i.e. equivalently
\begin{equation}\label{comparison-ineq}
\int_{0}^{T} m(\tau) \cdot e^{- \mu \tau} \, d\tau \geq m(T)\cdot (1-e^{-\mu T})/\mu.
\end{equation}
then the estimate \eqref{technical-ineq} implies \eqref{technical-ineq-alternative}. Observe that the \eqref{comparison-ineq} holds if $m$ is decreasing, e.g. in the parabolic case when $m(\eta)=M_\beta \eta^{-\beta}$ with $\beta\in (0,1)$. This will imply that our estimates provide stronger results for parabolic problems than those obtained in \cite{Rob-Vid-Lop} --  see Remark \ref{Rob-Vid-Lop-comp}. In another interesting case, with $m$ being a constant function, one has an equality between both sides of \eqref{comparison-ineq} and it is clear that
\eqref{technical-ineq} and \eqref{technical-ineq-alternative} provide the same results.
\end{uw}

\begin{proof}[Proof of Corollary \ref{cor-1}] We assume that the assumption \eqref{additional-as-main-abs-res} holds. If $T$ is the period of a nontrivial $T$-periodic solution of \eqref{abstract-ev-eq}, then either $T\geq 1/\mu_0 M$ or the inequality, coming from \eqref{technical-ineq}, holds 
$$
1  \leq T L  (1-\mu_0/\mu)^{-1}\cdot \left[ (1-\mu M T)^{-1}  + 1/\mu T\right],
$$
which, after setting $\eta= \mu M T$, yields
$$
1 - \mu_0 MT/\eta \leq T L\left( (1-\eta)^{-1} + M/\eta \right).
$$
Consequently 
\begin{equation}\label{per-est-H}
T \geq 1/G(\eta) 
\end{equation}
where $G:(0,1) \to (0,+\infty)$ is given by
$$
G(\eta):= L/(1-\eta) + LM/\eta + \mu_0 M/\eta = 
 L/(1-\eta) + C/\eta
$$
with $C:=M(L + \mu_0)$. A direct computation shows that $G$ attains the minimal value 
$$
G_0=(\sqrt{L}+\sqrt{C})^2 
$$
at the point $\eta_0 = \sqrt{C}/(\sqrt{L}+\sqrt{C})$.
Now observe that either $\eta_0/MT \leq \mu_0$,
which is equivalent to the inequality
$$
T\geq \eta_0/\mu_0 M = \sqrt{C}/\mu_0 M (\sqrt{L}+\sqrt{C})=
C/\mu_0 M(\sqrt{LC} +C)= (1+L/\mu_0)/(\sqrt{LC} +C) 
$$
or for $\mu= \eta_0/MT$ we get, by use of \eqref{per-est-H}, 
$$
T\geq 1/G(\eta_0) = 1/G_0 = 1/(\sqrt{L}+\sqrt{C})^2 = 1/(L+C+2\sqrt{LC}).
$$
Hence taking into consideration that 
$$
(1+L/\mu_0)/(\sqrt{LC} +C) \geq 1/(L+C+2\sqrt{LC})
$$
we see that $T\geq 1/(L+C+2\sqrt{LC}) = 1/L\left(1+\sqrt{M(1+\mu_0/L)}\right)^2.$
\end{proof}

\begin{uw}\label{remark-local-abs}
If instead of the global Lipschitzianity of ${\mathbb F}$ in Corollary \ref{cor-1} (comp.  Theorem \ref{main-res-abs}) we assume that for any $R>0$ there exists $L_R>0$ such that
$$
\|\mathbb{F}(z_1)-\mathbb{F}(z_2)\|_{\mathbb{X}}\leq L_R\|z_1-z_2\|_{\mathbb{V}} \ \ \mbox{ for all }\ \ z_1,z_2\in \{z\in\mathbb{X}\, \mid \, \|z\|_{\mathbb{V}} \leq R\},
$$
then we can slightly refine the assertion. Namely, if $z:\mathbb{R}\to {\mathbb V}$ is a $T$-periodic solution and 
$R:= \max\{\|z(t)\|\, \mid \, t\in\mathbb{R} \}$,
then, due to Corollary \ref{cor-1}, one has $T\geq 1/L_R\left(1+\sqrt{M(1+\mu_0/L_R)}
\right)^2$.
\end{uw}

\section{Parabolic equations -- proof of Theorem \ref{main-res-par}}

Assume that $A:D(A)\to X$ is as in Section 1, i.e. $A$ is a positive self-adjoint operator on a separable Hilbert space $X$ with the norm $\|\cdot \|$. Let $X^{\beta}$ with the fractional norm given by $\|u \|_{\beta} = \| A^{\beta}u\|$, $u\in X^{\beta}$, where $A^\beta$ is the fractional power of the operator $A$ (see e.g. \cite{Henry}). It is well-known that $-A$ generates an analytic $C_0$-semigroup $e^{-tA}$ such that $e^{-tA}(X)\subset D(A)$ and
$$
\| e^{-t A}\|_{{\cal L}(X,X^{\beta})} \leq M_\beta/t^{\beta} \ \ \mbox{ for all } u\in X^{\beta}
$$
with $M_\beta = (\beta/e)^{\beta}$ if $\beta\in(0,1)$ and $M_\beta=1$ if $\beta=0$. Let us collect below facts concerning spectral properties of such operators that can be obtained by use of spectral theory for self-adjoint operators (see Lemma 3.1 in \cite{Rob-Vid-Lop}). 
\begin{prop} Under the above assumptions, for any $\beta \in [0,1)$ and $\mu >0$, there exists a decomposition $X=X_{\mu}^{+} \oplus X_{\mu}^{-}$ with $X_{\mu}^{+}\subset D(A)$ and such that:\\
{\em (i)} \parbox[t]{144mm}{$\|A u\|\leq \mu \|u\|$, for all $u\in X_\mu^{+}$, and $\| P_{\mu}^{+} u \|_\beta \leq \mu^\beta \|u\|$, for all $u\in X$;}\\[0.5em]
{\em (ii)} \parbox[t]{144mm}{$\| e^{-tA} u \| \leq e^{-\mu t} \|u\|$, for all $u\in X_{\mu}^{-}$, and
$\| e^{-tA} u \|_{\beta} \leq e^{-\mu t} \|u\|_{\beta}$ for all $u\in X^{\beta}\cap X_{\mu}^{-}$;}\\[0.5em]
{\em (iii)} \parbox[t]{144mm}{$\| e^{-tA} u \|_{\beta} \leq M_{\beta}t^{-\beta} e^{-\mu t}  \|u\|$ for all $t>0$ and $u\in X_{\mu}^{-}$.}
\end{prop}
As an immediate consequence we get the following conclusion.
\begin{prop}\label{par-case-uhbd}
If $\mathbb{X}:=X$, $\mathbb{V}:=X^{\beta}$ with the fractional norm $\|\cdot \|_{\beta}$ and 
$\mathbb{A}:=-A$, then ${\mathbb A}$ satisfies the assumptions $(V_1)-(V_3)$ and has the $(UHBD)$ property with $\mu_0=0$, $M=1$, $K_{\mu}^{+} = \mu^\beta$, $K_{\mu}^{-}=1$, $m(t)= M_\beta/t^\beta$, $t>0$.
\end{prop}
The proof involves spectral calculus for positive self-adjoint operators in Hilbert spaces 

\noindent \begin{proof}[Proof of Theorem \ref{main-res-par}] 
Assume that $u\in C(\mathbb{R}, X^{\beta})$ is a nontrivial $T$-periodic solution of  \eqref{parabolic-ev-eq} with $T>0$ (it is clear that $u$ is also a mild solution). By Proposition \ref{par-case-uhbd}, we may apply Theorem \ref{main-res-abs} and obtain for all $\mu \in (0,1/T)$
\begin{align*}
1 & \leq    T L 
\cdot \left[ \frac{\mu^\beta}{1-\mu T} +  
\frac{M_\beta}{\mu T} \left(e^{-\mu T}/T^{\beta}  + \int_{0}^{\mu T} (s/\mu)^{-\beta} \cdot e^{- s} \, ds \right) \right]\\
& =  T^{1-\beta} L 
\cdot \left[ \frac{(\mu T)^\beta}{1-\mu T} +  
\frac{M_\beta}{\mu T} \left(e^{-\mu T}  + (\mu T)^{\beta}\int_{0}^{\mu T} s^{-\beta} \cdot e^{- s} \, ds \right) \right],
\end{align*}
i.e.
$$
1 \leq T^{1-\beta} L\cdot H(\mu T)
$$
where $H:(0,1)\to (0,+\infty)$ is given by \eqref{par-H-eta-def}. It is clear that  $1 \leq T^{1-\beta} \cdot L\cdot K_\beta$.\\
\indent Now in order to prove \eqref{1/2-better-estimate} taking $\mu>0$ such that $\mu T=1/2$ we get $1\leq T^{1-\beta} L \cdot H(1/2)$. Finally note that, for $\beta \in (0,1)$,
\begin{align}
H(1/2) & \leq 2^{1-\beta} + 2M_\beta \left( e^{-1/2} + (1/2)^{\beta} \int_{0}^{1/2} s^{-\beta} \, ds\right)\\
& =  2^{1-\beta} + 2M_\beta/ e^{1/2} + M_\beta/(1-\beta),
\end{align}
which implies \eqref{1/2-better-estimate}. If $\beta=0$, then $H$ attains minimum at $\eta=1/2$ and
$$
1\leq T L \cdot H(1/2) = T \cdot 4 L, 
$$
which ends the proof.
\end{proof}
\begin{uw}\label{Rob-Vid-Lop-comp}
(i) If we apply in the above proof the inequality \eqref{technical-ineq-alternative} instead of \eqref{technical-ineq}, then
we get
\begin{align*}
1 & \leq    T L 
\cdot \left( \frac{\mu^\beta}{1-\mu T} +  
\frac{M_\beta}{\mu T (1-e^{-\mu T})}  \int_{0}^{\mu T} (s/\mu)^{-\beta} \cdot e^{- s} \, ds \right)\\
& =  T^{1-\beta} L 
\cdot \left( \frac{(\mu T)^\beta}{1-\mu T} +  
\frac{M_\beta}{(\mu T)^{1-\beta}(1-e^{-\mu T})} \int_{0}^{\mu T} s^{-\beta} \cdot e^{- s} \, ds \right),\\
& = T^{1-\beta}L\cdot\tilde H(\mu T),
\end{align*}
where $\tilde H:(0,1)\to (0,+\infty)$ is given by
$$
\tilde H(\eta):= \frac{\eta^\beta}{1-\eta} +  
\frac{M_\beta}{\eta^{1-\beta}(1-e^{-\eta})} \int_{0}^{\eta} s^{-\beta} \cdot e^{- s} \, ds.
$$
Observe that 
\begin{align} \nonumber
\tilde H(1/2) & = 2^{1-\beta} +  
\frac{M_\beta}{(1/2)^{1-\beta}(1-e^{-1/2})} \int_{0}^{1/2} s^{-\beta} \cdot e^{- s} \, ds \\
& <  2^{1-\beta} +  
\frac{M_\beta}{(1/2)^{1-\beta}(1-e^{-1/2})} \int_{0}^{1/2} s^{-\beta} \, ds
= 2^{1-\beta} + M_\beta/(1-e^{-1/2})(1-\beta), \label{est-better-01}
\end{align}
which allows us to deduce the estimate \eqref{Robinson-Vidal-Lopez-estimate} that was provided originally in 
\cite{Rob-Vid-Lop}.\\
\indent (ii) The estimate \eqref{main-res-est} is stronger than \eqref{Robinson-Vidal-Lopez-estimate}. Indeed, reasoning as in Remark \ref{alternative-procedure}, we see that $\tilde H(\eta)  >  H(\eta)$ for all $\eta\in (0,1)$. Therefore, in view of \eqref{est-better-01},  
$$
2^{1-\beta} + M_\beta/(1-e^{-1/2})(1-\beta)> \tilde H(1/2) \geq \min_{\eta\in (0,1)} \tilde H(\eta) > K_\beta,
$$ 
which shows the relation between the estimates.\\
\indent (iii) To see that also \eqref{1/2-better-estimate} is stronger that \eqref{Robinson-Vidal-Lopez-estimate} one can directly verify the inequality
$$
2^{1-\beta} + M_\beta/(1-e^{-1/2})(1-\beta) > 2^{1-\beta} + 2M_\beta/e^{1/2} + M_\beta/(1-\beta)
$$
that is equivalent to $4>e$.
\end{uw}

\section{Spectrum of hyperbolic operator}
Let us assume that $A:D(A)\to X$ is a sectorial operator in a Banach space $X$ and that the so-called {\em hyperbolic operator} $\bf{A} = \bold{A}_{\alpha}:D(\bold{A}_\alpha)\to \bold{X}$ in $\bold{X}:= X^{1/2} \times X$, where $X^{1/2}$ is the fractional space related to $A$, is defined by
\begin{equation}\label{hyper-op}
{\bold A} (u,v) :=  (v, -A (\alpha \cdot v+u) ), \ (u,v) \in D({\bold A}).
\end{equation}
with
$$
D({\bold A}) := \{ (u,v) \in {\bold X}\, \mid\, \alpha \cdot v+u \in D(A),\, v\in X^{1/2} \}.
$$
Without loss of generality we may assume that $X$ is a complex space, so is $\bold{X}$. Let us start with the following observation.
\begin{lm}\label{basic_lemma}
Suppose that $(u,v), (g,h) \in {\bold X}$ and $\xi \in \mathbb{C}$. 
If $\xi \neq -1 /\alpha$ then
\begin{equation}\label{res-eq}
(u,v)\in D({\bold A}) \ \ \mbox{ and } \ \ (\xi I - {\bold A}) (u,v) = (g,h)
\end{equation}
if and only if $w:=u+\alpha \cdot v\in D(A)$,
\begin{equation}\label{u-v-eqs}
u = \frac{1}{1+\alpha \xi} \cdot w+ \frac{\alpha}{1+\alpha \xi} \cdot g, \  \ \ \  v = \frac{\xi}{1+\alpha \xi} \cdot w -  \frac{1}{1+\alpha \xi} \cdot g
\end{equation}
and
\begin{equation}\label{A-res-eq}
(A-s(\xi)I) w =  \frac{\xi}{1+\alpha \xi} \cdot g  + h
\end{equation}
where the mapping $s:\mathbb{C}\setminus \{-1/\alpha\}\to \mathbb{C}$ is given by
$$
s(\xi):=  - \ \frac{\xi^2}{1+\alpha \xi}.
$$
If $\xi = -1/\alpha$ then \eqref{res-eq} is equivalent to the following condition
\begin{equation}
g\in D(A), \ \ \ u = -\alpha\cdot(-\alpha^2\cdot Ag + g - \alpha \cdot h)   \ \ \ \mbox{ and  } \ \ \ v =-\alpha\cdot(\alpha\cdot Ag+h).
\end{equation}
\end{lm}
\begin{proof} Suppose that \eqref{res-eq} holds. Then, by the definition of ${\bold A}$, $w=u+\alpha \cdot v\in D(A)$ and 
\begin{equation}\label{A-sys-of-eqs}
\xi \cdot u - v = g  \ \ \  \mbox{ and  }\ \ \  \xi \cdot v + A w = h.
\end{equation}
Substituting $u= w-\alpha\cdot v$ in the first equality of \eqref{A-sys-of-eqs}  one gets 
\begin{equation}\label{what-if}
\xi \cdot w - (1+\alpha \xi) \cdot v = g
\end{equation}
and
$v= \xi (1+\alpha \xi)^{-1} \cdot w - (1+\alpha\xi)^{-1} \cdot g$, i.e. the second part of 
\eqref{u-v-eqs}. Applying it in $u=w-\alpha \cdot v$, one has $u=(1+\alpha \xi)^{-1}\cdot w+\alpha(1+\alpha\xi)^{-1}\cdot g$, i.e. the first part of \eqref{u-v-eqs}, and applying second part of \eqref{u-v-eqs} in the second equality of \eqref{A-sys-of-eqs} we arrive at \eqref{A-res-eq}.\\
\indent Now suppose that $w=u+\alpha \cdot v\in D(A)$ and  both \eqref{u-v-eqs} and  \eqref{A-res-eq} hold. Then, obviously $v= \alpha^{-1} \cdot (w-u) \in X^{1/2}$ and  a direct calculation shows that
$$
(\xi I - {\bold A})(u,v) = (\xi \cdot u - v,\ \ \xi \cdot v+Aw)= (g,\ -s(\xi)\cdot w -
\xi(1+\alpha\xi)^{-1}\cdot g +Aw) = (g,h),
$$
which ends the proof of the first equivalence.\\
\indent The proof of the second one is similar and therefore omitted. 
\end{proof}
\begin{uw}\label{resolvent-verification}
Suppose $B:D(B)\to X$ is an arbitrary closed operator in a Banach space $X$. Recall that $0\in \rho(B)$ if and only if for any $g\in X$ there exists a unique $u\in D(B)$ such that $Bu=g$.
\end{uw}
The next result provides the characterization  of the spectrum of ${\bold A}$.
\begin{prop}\label{thm_general_properties_of_A}
The operator ${\bold A}$ defined by \eqref{hyper-op} has the following properties
\begin{enumerate}[{\em (i)}]
\item $\bold A$ is closed.
\item $\rho (\bold A) \setminus \left\{ - 1/\alpha \right\} = s^{- 1} (\rho (A))$.
\item $\sigma (\bold A) \setminus \left\{ - 1/\alpha \right\} = s^{- 1}(\sigma (A))$.
\item If $\xi \neq -1/\alpha$, then
$\Ker (\xi \bold I - \bold A) = \left\{ (u,\xi \cdot u)\ | \  u \in \Ker (s (\xi) I - A)\right\}$.
\end{enumerate}
\end{prop}
\begin{proof}
(i) Suppose  $(u_n,v_n) \in D({\bold A})$, $n\geq 1$, and $(u_n,v_n)\to (u,v)$ and ${\bold A}(u_n, v_n)\to (g,h)$ in ${\bold X}$. It follows directly that $u_n \to u$ in $X^{1/2}$, $v_n\to v$ in $X$ and that $v_n \to g$ in $X^{1/2}$ and $-A(u_n+\alpha \cdot v_n)\to h$, which implies $v=g \in X^{1/2}$ and that $u_n + \alpha\cdot v_n \to u+\alpha \cdot v$. By the closedness of $A$, $u+\alpha\cdot v \in D(A)$ and $A(u+\alpha\cdot v)=-h$. It means that $(u,v)\in D(\bold A)$ and ${\bold A} (u,v) = (v,-A(u+\alpha \cdot v))=(g,h)$, which shows (i).\\
\indent (ii) If $\xi\in \rho (\bold A) \setminus \left\{ - 1/\alpha \right\}$, then, in particular, for $g=0$ and any $h\in X$,  there is a unique $(u,v)\in D(\bold A)$ such that $(\xi I - {\bold A})(u,v)=(0,h)$, which due to Lemma \ref{basic_lemma} means that $w:=u+\alpha \cdot v\in D(A)$ is the only solution of $(A-s(\xi)I)w=h$, which, in view of Remark \ref{resolvent-verification} means that $s(\xi)\in \rho(A)$. On the other hand, if $s(\xi)\in \rho(A)$ and we take any $(g,h)\in {\bold X}$, then we get a unique $w\in D(A)$ satisfying \eqref{A-res-eq}. Then $(u,v)$ given by \eqref{u-v-eqs} solves $(\xi I - {\bold A})(u,v)=(g,h)$. The uniqueness of solutions comes immediately from the uniqueness of $w$. Hence, again by Remark \ref{resolvent-verification}, we infer that $\xi\in \rho({\bold A})$.\\    
\indent (iii) Observe that, by use of (ii), 
\begin{align*}
\sigma({\bold A})\setminus \{-1/\alpha\} & = 
\left( \mathbb{C} \setminus \rho({\bold A}) \right)  \setminus\{-1/\alpha\} =\\
& = \left( \mathbb{C}\setminus\{-1/\alpha\} \right)  
\setminus \left( \rho({\bold A})  \setminus\{-1/\alpha\}
\right)
= \left( \mathbb{C}\setminus\{-1/\alpha\}\right) \setminus s^{-1}(\rho(A))\\
& = s^{-1} (\mathbb{C}\setminus \rho(A))=
s^{-1}(\sigma(A)).
\end{align*}
\indent (iv) follows directly from Lemma \ref{basic_lemma}.
\end{proof}
\begin{wn}\label{cor_properties_of_A-characterization}
If $A$ is an unbounded operator with $\sigma(A)\subset (0,+\infty)$, then
$$
\sigma({\bold A})\setminus\left\{-1/\alpha\right\} = \xi_-(\sigma(A)) \cup \xi_+(\sigma(A))
$$
and
$$
\sigma_p ({\bold A}) = \xi_- (\sigma_p (A)) \cup \xi_+ (\sigma_p(A))
$$
where if $\lambda>(2/\alpha)^2$
$$
\xi_{-}(\lambda)=\frac{-\alpha \lambda -\sqrt{\Delta(\lambda)} }{2}=- \frac{2}{\alpha} \cdot \frac{1}{1-\sqrt{1-(2/\alpha)^2/\lambda}},
$$
$$
\xi_{+}(\lambda)=\frac{-\alpha \lambda +\sqrt{\Delta(\lambda)} }{2}=- \frac{2}{\alpha} \cdot \frac{1}{1+\sqrt{1-(2/\alpha)^2/\lambda}},
$$
if $0<\lambda<(2/\alpha)^2$
$$
\xi_{-}(\lambda)=\frac{-\alpha \lambda -i \sqrt{-\Delta(\lambda)} }{2}=- \frac{2}{\alpha} \cdot \frac{1}{1-i\sqrt{(2/\alpha)^2/ \lambda-1}},
$$
$$
\xi_{+}(\lambda)=\frac{-\alpha \lambda +i \sqrt{-\Delta(\lambda)} }{2}=- \frac{2}{\alpha} \cdot \frac{1}{1+i\sqrt{(2/\alpha)^2/ \lambda-1}}
$$
and if $\lambda=(2/\alpha)^2$
$$
\xi_{-}(\lambda)=\xi_{+}(\lambda)=\frac{-\alpha\lambda}{2},
$$
with $\Delta(\lambda)=(\alpha\lambda)^2 -4\lambda$. Moreover
\begin{enumerate}[{\em (i)}]
\item For any $\lambda > 0$,  $\xi_{-}(\lambda)$ and $\xi_{+}(\lambda)$ are the roots of 
\begin{equation}\label{polynomial}
\xi^2 + \alpha \lambda \xi + \lambda=0,
\end{equation}
in particular
$$
\xi_-(\lambda)+\xi_+(\lambda)= - \alpha\lambda  \ \ \mbox{ and } \ \  \xi_-(\lambda)\cdot \xi_+(\lambda)= \lambda.
$$
\item $\sigma ( {\bold A} ) \subset \{z\in\mathbb{C}\,\mid \, \mathrm{Re}\, z<0\}$.
\item If $\lambda> (2/\alpha)^2$ then
$$
\xi_-(\lambda)<-2/\alpha < \xi_+ (\lambda) < -1/\alpha.
$$
\item If $0<\lambda < (2/\alpha)^2$  then
$$
-2/\alpha < \mathrm{Re}\, \xi_\pm(\lambda) < 0.
$$
\item $\xi_{-} (\lambda)\to -\infty$ as $\lambda\to \infty$ and $\xi_-$ is decreasing on $\left( (2/\alpha)^2, +\infty\right)$.
\item $\xi_+ (\lambda) \to -1/\alpha$ as $\lambda \to \infty$ and $\xi_+$ is increasing on $\left( (2/\alpha)^2, +\infty\right)$.
\end{enumerate}
\end{wn}
\begin{proof}
Observe that $s^{-1}(\{\lambda\})$ consists of $\xi \in \mathbb{C}$ solving \eqref{polynomial},
i.e. $s^{-1}(\{\lambda\}) = \{\xi_-(\lambda), \xi_+(\lambda) \}$.\\
\indent Assertions (i)-(vi) are immediate.
\end{proof}

In order to estimate the norms of projections and ${\bold A}$ on eigenspaces we shall need the following elementary fact.
\begin{lm}\label{elementary-max}
For any $z\in\mathbb{C}$ one has the following equalities
$$
\max\{|z \cdot z_1 + z_2| \, \mid\, z_1, z_2\in \mathbb{C},\,    |z_1|^2 + |z_2|^2=1 \} =\sqrt{1+|z|^2};    \leqno{(i)}
$$
$$
\max\{|z_1|^2+ |2\cdot z \cdot z_1+z_2|^2\, \mid\, z_1, z_2\in \mathbb{C},\, |z_1|^2+|z_2|^2=1 \} = (|z|+\sqrt{1+|z|^2})^2. \leqno{(ii)}
$$
\end{lm}
The following proposition will be crucial when computing the norms of projection onto spectral decomposition components.
\begin{prop}\label{one-dim-projections}
Suppose that $\lambda\in \sigma_p(A) \setminus \{(2/\alpha)^2\}$ and $e_0$ be element of $\mathrm{Ker}\, (A-\lambda I)$ with $\|e_0\|=1$. Then, for any $u,v\in \mathbb{C}$,
\begin{equation}\label{change-of-variable} 
u\cdot (e_0,0)+v \cdot (0,e_0) = p_{\lambda}^-(u,v)\cdot {\bold e}_{\lambda}^{-} + p_{\lambda}^+ (u,v) \cdot {\bold e}_{\lambda}^{+}
\end{equation}
with ${\bold e}_{\lambda}^{-}:= (e_0,\xi_-(\lambda) \cdot e_0)$,
${\bold e}_{\lambda}^{+}:=(e_0, \xi_{+}(\lambda) \cdot e_0 )$, being the eigenvalues of ${\bold A}$ corresponding to $\xi_+(\lambda)$ and $\xi_-(\lambda)$, respectively, and
$$
p_{\lambda}^{-} (u,v):= \frac{ \xi_{+}(\lambda)u-v}{\xi_{+}(\lambda)-\xi_{-}(\lambda)}, \ \ \ \ \ p_{\lambda}^{+} (u,v) := \frac{\xi_{-}(\lambda)u-v}{\xi_{-}(\lambda)-\xi_{+}(\lambda)}.
$$
Let ${\bold X}_{e_0}:= \mathrm{span}\left\{(e_0,0),(0,e_0)\right\}$,  ${\bold P}_{\lambda}^{-}:{\bold X}_{e_0}  \to \mathbb{C}\cdot {\bold e}_{\lambda}^{-}$ and ${\bold P}_{\lambda}^{+}:{\bold X}_{e_0} \to \mathbb{C}\cdot {\bold e}_{\lambda}^{+}$ be the projections, i.e.
$$
{\bold P}_{\lambda}^{-} (u\cdot e_0,v\cdot e_0)=p_{\lambda}^{-} (u,v) \cdot {\bold e}_{\lambda}^{-},  \ \ \ \ \ 
{\bold P}_{\lambda}^{+} (u\cdot e_0,v\cdot e_0)=p_{\lambda}^{+} (u,v) \cdot {\bold e}_{\lambda}^{+}
$$
If $\lambda>(2/\alpha)^2$, then
\begin{equation}
\|{\bold P}_{\lambda}^{-}\| = \|{\bold P}_{\lambda}^{+}\| = 1/\sqrt{1-(2/\alpha)^2/\lambda}
\end{equation}
and, if $0<\lambda< (2/\alpha)^2$, then
\begin{equation}
\|{\bold P}_{\lambda}^{-}\| = \|{\bold P}_{\lambda}^{+}\| = 1/\sqrt{1-\lambda/(2/\alpha)^2}.
\end{equation}
\end{prop}
\begin{proof}
The equality \eqref{change-of-variable} can be verified by a direct algebraic computation.\\ 
\indent By use of Lemma \ref{elementary-max} (i) one has
\begin{align*}
\| {\bold P}_{\lambda}^{-} \|& = \max\left\{ \| {\bold P}_{\lambda}^{-}(u\cdot e_0, v\cdot e_0)\|_{\bold X}\, \mid  
\|(u\cdot e_0, v\cdot e_0)\|_{\bold X} = 1 \right\}\\
& = \max\left\{ |p_{\lambda}^{-} (u,v)|\cdot\|(e_0,\xi_-(\lambda) \cdot e_0)\|_{\bold X}  \, \mid\, \lambda |u|^2+|v|^2=1 \right\}\\
& =\max\left\{ |(\xi_+(\lambda)/\sqrt{\lambda}) z_1 - z_2| \, \mid \, |z_1|^2+|z_2|^2=1 \right\} \cdot \sqrt{\lambda +|\xi_-(\lambda)|^2} \cdot |\xi_+(\lambda)-\xi_-(\lambda) |^{-1} \\
& =  \sqrt{1+ |\xi_{+}(\lambda)|^2/\lambda}\cdot \sqrt{\lambda +|\xi_-(\lambda)|^2} \cdot |\xi_+(\lambda)-\xi_-(\lambda) |^{-1}.
\end{align*}
Hence, if $\lambda> (2/\alpha)^2$ then $\xi_-(\lambda)$, $\xi_+(\lambda)\in\mathbb{R}$ and, in view of (i) in Corollary \ref{cor_properties_of_A-characterization}, one gets  
\begin{align*}
\| {\bold P}_{\lambda}^{-} \|^2 & =
(1+ \xi_{+}(\lambda)^2/\lambda)\cdot (\lambda +\xi_-(\lambda)^2)/ (\xi_+(\lambda)-\xi_-(\lambda) )^{2}\\
& = \lambda^{-1} \cdot (\xi_+(\lambda)\xi_-(\lambda) + \xi_+(\lambda)^2)\cdot (\xi_+(\lambda)\xi_-(\lambda) + \xi_-(\lambda)^2) / \Delta(\lambda)\\
& = (\xi_-(\lambda) +\xi_+(\lambda))^2/ \Delta(\lambda)= (\alpha \lambda)^2 / ((\alpha\lambda)^2-4\lambda)=
\frac{1}{1-(2/\alpha)^2/\lambda}.
\end{align*}
In the same way one computes the norm of ${\bold P}_{\lambda}^{+}$. When $0 < \lambda < (2/\alpha)^2$, then $\xi_+(\lambda) = \overline{\xi_-(\lambda)}$ and
\begin{align*}
\| {\bold P}_{\lambda}^{-} \|^ 2& =
(1+ |\xi_{+}(\lambda)|^2/\lambda)\cdot (\lambda +|\xi_-(\lambda)|^2)/ |\xi_+(\lambda)-\xi_-(\lambda) |^{2}\\
& = 4\lambda / (4\lambda - (\alpha\lambda)^2) =
\frac{1}{1-\lambda/(2/\alpha)^2}.
\end{align*}
As before, the computation for ${\bold P}_{\lambda}^{+}$ is similar. 
\end{proof}
\begin{prop}\label{two-dim-operators}
Suppose that $\lambda\in\sigma_p(A)$, $e_0\in\mathrm{Ker}\, (A-\lambda I)$ with $\|e_0\|=1$ and ${\bold X}_{e_0}:=\mathrm{span}\left\{(e_0,0),(0,e_0)\right\}$ and let ${\bold A}_{e_0}:{\bold X}_{e_0}\to {\bold X}_{e_0}$ be the restriction of ${\bold A}$ to ${\bold X}_{e_0}$.
Then
$$
\|{\bold A}_{e_0}\| = (2/\alpha) \cdot g\left( (2/\alpha)^2/\lambda \right)
$$
where $g:(0,+\infty)\to (0,+\infty)$ is given by
$$
g(r):= \frac{1+\sqrt{1+r}}{r}, \, r>0.
$$
\end{prop}
\begin{proof}
Observe that, for any $u,v\in \mathbb{C}$,
$$
{\bold A}(u\cdot e_0,v\cdot e_0) = (v\cdot e_0, -A(u\cdot e_0 + \alpha v\cdot e_0))=
(v\cdot e_0, -\lambda(u+\alpha v)\cdot e_0).
$$
Hence, by use of Lemma \ref{elementary-max} (ii),
\begin{align*}
\|{\bold A}_{e_0}\|^2 & = \max \{ \|{\bold A}(u\cdot e_0,v\cdot e_0) \|_{\bold X}^{2} \mid \| (u\cdot e_0,v\cdot e_0)\|_{\bold X} = 1 \} \\ 
&  = \lambda \cdot \max\{ |v|^2 + \lambda |u+\alpha v|^2 \mid \lambda |u|^2 + |v|^2 = 1\} \\
& = \lambda \cdot \max \{ |z_1|^2 + |\lambda^{1/2} \alpha z_1 + z_2|^2 \, \mid \, |z_1|^2 +|z_2|^2=1  \}\\
& = \lambda (\lambda^{1/2}\alpha/2 + \sqrt{1+\lambda (\alpha/2)^2})^2\\
& = (2/\alpha)^2 \left( \lambda(\alpha/2)^2 + \sqrt{\lambda(\alpha/2)^2 + \lambda^2(\alpha/2)^4} \right)^2 = (2/\alpha)^2 \cdot
\left(r^{-1} + \sqrt{r^{-1}+r^{-2}}\right)^2
\end{align*}
where $r= (2/\alpha)^2/\lambda$.
\end{proof}
\section{Spectral decomposition and property $(UHBD)$ for hyperbolic operator}
Assume that $A:D(A)\subset X \to X$ is a self-adjoint operator in a separable Hilbert space $X$ (endowed with the scalar product $\langle .,.\rangle$ and norm $\|.\|$) with the spectrum $\sigma(A)$ consisting of positive eigenvalues $\lambda_k$, $k\geq 1$, of finite multiplicities such that $(\lambda_k)$ is nondecreasing and $\lambda_k \to \infty$ as $k\to \infty$. The corresponding eigenvectors we denote by $e_k$, $k\geq 1$. Then it is clear (see Theorem 13.36 in \cite{Rudin}) that $D(A)\neq X$ and therefore, in view of Lemma \ref{basic_lemma}, $-1/\alpha \in \sigma({\bold A})$. In view of Corollary \ref{cor_properties_of_A-characterization} we infer that
$$
\sigma (\bold A)  = 
\left\{ -1/\alpha \right\} \cup \bigcup_{k = 1}^\infty \left\{ \xi_{k}^{-},\ \xi_{k}^{+} \right\}
$$
where $\xi_{k}^{-}:=\xi_-(\lambda_k)$, $\xi_{k}^{+}:=\xi_+ (\lambda_k)$.  It is also clear that if $\xi_{k}^{-}$, $\xi_{k}^{+} \in \mathbb{R}$ then
$$
\xi_{k}^{-} < - 2/\alpha < \xi_{k}^{+} < -1/\alpha,
$$
$\xi_{k}^{-} \rightarrow  - \infty$  and $\xi_{k}^{+} \rightarrow -1/\alpha$. Moreover, the set $\sigma({\bold A})\setminus \mathbb{R}$ is finite and
$$
\sigma(\bold A) \setminus \mathbb{R} \subset
\{z\in\mathbb{C}\,\mid\,  -2/\alpha < \mathrm{Re}\, z< 0\}.
$$
Using the notation from Proposition \ref{one-dim-projections} we define
$$
{\bold e}_{k}^{-}:= {\bold e}_{\lambda_k}^{-} = (e_k, \xi_{k}^{-} \cdot e_k), \ \ \ 
{\bold e}_{k}^{+}:= {\bold e}_{\lambda_k}^{+} = (e_k, \xi_{k}^{+} \cdot e_k), \ \ \  k\in \mathbb{N}\setminus N^0,
$$
where  $N^{0} := \{k\in \mathbb{N}\, \mid\, \lambda_k = (2/\alpha)^2\}$ (it is clearly a finite set) and 
$$
{\bold e}_{k}^{-}:= (e_k, 0),  \ \ \ {\bold e}_{k}^{+}:=(0,e_k), \  \ \ k\in N^0. 
$$
Observe that, in view of \eqref{change-of-variable},
\begin{align*}
{\bold X} & = \overline{\bigoplus_{k\in \mathbb{N}} \  \mathrm{span}\, \left\{(e_k,0), (0,e_k)\right\}} = \overline{\bigoplus_{k\in \mathbb{N}\setminus N^0}
\mathrm{span}\, \left\{ {\bold e}_{k}^{-}, {\bold e}_{k}^{+} \right\}} \ \oplus {\bold X}^0
\end{align*}
and
$$
{\bold X}^{0}:= \bigoplus_{k\in N^0} \mathrm{span}\,\left\{ {\bold e}_{k}^{-}, {\bold e}_{k}^{+} \right\}.
$$
Obviously, for any $k, l \in \mathbb{N}$ such that $k\neq l$,
$$
\mathrm{span} \left\{ {\bold e}_{k}^{-}, {\bold e}_{k}^{+} \right\} 
\perp
\mathrm{span} \left\{ {\bold e}_{l}^{-}, {\bold e}_{l}^{+} \right\}.   
$$
\begin{prop}\label{series-in-boldX}
For any $(u,v)\in {\bold X}$
$$
(u,v) = \sum_{k=1}^{\infty} p_{k}^{-}(u_k, v_k)\cdot {\bold e}_{k}^{-}
+ \sum_{k=1}^{\infty} p_{k}^{+}(u_k, v_k)\cdot{\bold e}_{k}^{+} \mbox{ in } {\bold X}
$$
where  $p_{k}^{-}:=p_{\lambda_k}^{-}$, 
$p_{k}^{+}:=p_{\lambda_k}^{+}$ if $k\in \mathbb{N} \setminus N^0$ and $p_{k}^{-}(z_1,z_2):=z_1$, $p_{k}^{+}(z_1, z_2):=z_2$ if $k\in N^0$ and for any $k\in\mathbb{N}$ $u_k:=\langle u,e_k\rangle$, $v_k:=\langle v,e_k\rangle$.
In consequence
$$
{\bold X}  =  \overline{\bigoplus_{k\in \mathbb{N}\setminus N^0} \mathbb{C}\cdot {\bold e}_{k}^{-}} \ \oplus \ \overline{\bigoplus_{k\in \mathbb{N}\setminus N^0} \mathbb{C}\cdot {\bold e}_{k}^{+}} \ \oplus \  {\bold X}^{0}. 
$$
\end{prop}
\begin{proof}
Applying Proposition \ref{one-dim-projections}, for large $k$, we obtain
\begin{align*}
\|p_{k}^{-} (u_k, v_k) \cdot{\bold e}_{k}^{-} \|_{\bold X}^{2} & \leq 
1/(1-(2/\alpha)^2/\lambda_k)\cdot \|(u_k \cdot e_k,v_k \cdot e_k)\|_{\bold X}^{2}\\
& = 1/(1-(2/\alpha)^2/\lambda_k) (\lambda_k |u_k|^2+|v_k|^2).
\end{align*}
The same estimate we get for $\|p_{k}^{+} (u_k, v_k)\cdot {\bold e}_{k}^{+} \|_{\bold X}^{2}$. Since, $1/(1-(2/\alpha)^2/\lambda_k)\to 1$ as $k\to\infty$ and  the series with terms $\lambda_k |u_k|^2 + |v_k|^2$ is convergent, we see that the series with the terms $\|p_{k}^{-} (u_k, v_k) \cdot{\bold e}_{k}^{-} \|_{\bold X}^{2}$ and $\|p_{k}^{+} (u_k, v_k)\cdot {\bold e}_{k}^{+} \|_{\bold X}^{2}$ are convergent as well.
\end{proof}
Now take any $\mu > 2/\alpha$ and consider the spectral decomposition into 
$$
{\bold X} = {\bold X}_{\mu}^{-}  \oplus  {\bold X}_{\mu}^{+} 
$$
corresponding to the spectral sets (see \cite{Dunford-Schwartz})
$\sigma_{\mu}^{-}$ and $\sigma_{\mu}^{+}$  given by
$$
\sigma_{\mu}^{-}:= \{z\in \sigma({\bold A}) \,\mid\,  \mathrm{Re}\, z \leq - \mu \}, \ \   \sigma_{\mu}^{+}:= \sigma({\bold A}) \setminus \sigma_{\mu}^{-}. 
$$
The following proposition, being a straightforward consequence of Lemma \ref{basic_lemma}, Corollary \ref{cor_properties_of_A-characterization}
and Proposition \ref{series-in-boldX}, provides explicitly the components of the above decomposition.
\begin{prop}
If $\mu>2/\alpha$ then
\begin{align*}
{\bold  X}_\mu^{-}  &  =   \overline{\bigoplus\limits_{k\in N_\mu^{-}} \mathbb{C} \cdot
{\bold e}_{k}^{-}},\\
{\bold  X}_\mu^{+} & = \bigoplus\limits_{k\in \mathbb{N}\setminus (N_\mu^{-}\cup N^0)}    
\mathbb{C} \cdot {\bold e}_{k}^{-} \oplus \overline{\bigoplus\limits_{k\in \mathbb{N}\setminus N^0 }\mathbb{C} \cdot {\bold e}_{k}^{+}} \oplus {\bold X}^{0}. 
\end{align*}
where 
$N_{\mu}^{-} :=\{ k\in \mathbb{N} \,\mid \, \lambda_k>(2/\alpha)^2,\, \xi_{k}^{-}\leq -\mu \}.$
\end{prop}
Next we shall estimate the norms of the projections ${\bold P}_{\mu}^{-}$ and ${\bold P}_{\mu}^{+}$ in ${\bold X}$ onto ${\bold X}_{\mu}^{-}$ and ${\bold X}_{\mu}^{+}$, respectively.
\begin{prop}\label{to-obtain-(D)-property-1}
If $\mu>2/\alpha$ then
$$
\|{\bold P}_{\mu}^{-}\|_{{\cal L}({\bold X}, {\bold X})} \leq  (1- (2/\alpha)/\mu)^{-1}  \ \mbox{ and } \ 
\|{\bold P}_{\mu}^{+}\|_{{\cal L}({\bold X}, {\bold X})} \leq  (1- (2/\alpha)/\mu)^{-1}.
$$
\end{prop}
\begin{proof}
Let ${\bold P}_{k}^{-} := {\bold P}_{\lambda_k}^{-} \circ 
{\bold P}_k$ where ${\bold P}_{\lambda_k}^{-}:\mathrm{span} \left\{(e_k,0),  (0,e_k) \right\} \to \mathbb{C}\cdot {\bold e}_{k}^{-}$ as in Proposition \ref{one-dim-projections} and 
${\bold P}_k:{\bold X}\to \mathrm{span} \left\{(e_k,0),  (0,e_k) \right\}$ is the projection defined as follows: for any $(u,v)\in \bold X$ ${\bold P}_k(u,v):=(u_k\cdot e_k,v_k\cdot e_k)$. Since all of the one dimensional components of ${\bold X}_{\mu}^{-}$ are orthogonal with respect to each other and 
$$
1/\sqrt{1-(2/\alpha)^2/\lambda_k } \leq 1/(1-(2/\alpha)/\mu) \mbox{ for } k\in N_{\mu}^{-},
$$
we infer that, for any $(u,v)\in {\bold X}$,
\begin{align*}
\|{\bold P}_{\mu}^{-} (u,v)\|^2 & =  \sum_{k\in N_{\mu}^{-}} \|{\bold P}_{k}^{-}(u,v)\|^2 \leq \sum_{k\in N_{\mu}^{-}}
1/(1-(2/\alpha)^2/\lambda_k) \cdot \|(u_k\cdot e_k,v_k\cdot e_k) \|_{\bold X}^{2}\\
& \leq 1/(1-(2/\alpha)/\mu)^2 \sum_{k\in N_{\mu}^{-}} \|(u_k\cdot e_k,v_k\cdot e_k) \|_{\bold X}^{2} \leq 1/(1-(2/\alpha)/\mu)^2 \cdot \|(u,v)\|_{\bold X}^{2}.
\end{align*}
In order to estimate the norm of ${\bold P}_{\mu}^{+}$ observe that ${\bold X}_{\mu}^{+}$ can be split into four orthogonal parts
\begin{equation}\label{spliting-mu-plus}
{\bold X}_{\mu}^{+} = {\bold X}_{\mu,\mathbb{R}}^{+}
\oplus {\bold X}_{\mu,\mathbb{C}}^{+} \oplus {\bold X}_{\mu,-}^{+} \oplus {\bold X}^{0}
\end{equation}
where
$$
{\bold X}_{\mu,\mathbb{R}}^{+}:= \bigoplus_{k\in N_{\mu,\mathbb{R}}^{+}}
\mathrm{span}\left\{ {\bold e}_{k}^{-},{\bold e}_{k}^{+} \right\}, \ \ \  \
{\bold X}_{\mu,\mathbb{C}}^{+} := \bigoplus_{k\in N_{\mu,\mathbb{C}}^{+}}
\mathrm{span}\left\{ {\bold e}_{k}^{-},{\bold e}_{k}^{+} \right\}, \ \ \ \
{\bold X}_{\mu,-}^{+}:= \overline{\bigoplus_{k\in N_{\mu}^{-}} \mathbb{C}\cdot {\bold e}_{k}^{+}}
$$
with 
\begin{align*}
N_{\mu,\mathbb{R}}^{+} & := \{ k\in \mathbb{N} \, \mid\,  \lambda_k > (2/\alpha)^2,\, \xi_{k}^{-}> -\mu \},\\
N_{\mu, \mathbb{C}}^{+}& := \{ k\in \mathbb{N} \, \mid\,  0<\lambda_k < (2/\alpha)^2  \}.
\end{align*}
Clearly, in view of the orthogonality of these components, for any $(u,v)\in {\bold X}$,
\begin{equation}\label{ortho-sum-proj-plus}
\|{\bold P}_{\mu}^{+} (u,v)\|_{\bold X}^{2}
= \|{\bold P}_{\mu, \mathbb{R}}^{+} (u,v)\|_{\bold X}^{2} +\|{\bold P}_{\mu, \mathbb{C}}^{+} (u,v)\|_{\bold X}^{2}+\|{\bold P}_{\mu,-}^{+} (u,v)\|_{\bold X}^{2} + \|{\bold P}^{0} (u,v)\|_{\bold X}^{2} 
\end{equation}
where ${\bold P}_{\mu, \mathbb{R}}^{+}$, 
${\bold P}_{\mu, \mathbb{C} }^{+}$, ${\bold P}_{\mu,-}^{+}$ 
and ${\bold P}^{0}$ are projections onto the proper components. Therefore
$$
\|{\bold P}_{\mu, \mathbb{R}}^{+} (u,v) \|_{\bold X}^{2}
= \sum_{k\in N_{\mu,\mathbb{R}}^{+}} \|(u_k\cdot e_k,v_k\cdot e_k)\|_{\bold X}^{2}, \ \ \ 
\|{\bold P}_{\mu, \mathbb{C}}^{+} (u,v)\|_{\bold X}^{2}
=\sum_{k\in N_{\mu,\mathbb{C}}^{+}} \|(u_k\cdot e_k,v_k\cdot e_k)\|_{\bold X}^{2}
$$
and
$$
\|{\bold P}^{0} (u,v)\|_{\bold X}^{2}
= \sum_{k\in N^0} \|(u_k\cdot e_k,v_k\cdot e_k)\|_{\bold X}^{2}.
$$
Reasoning as in the case of ${\bold P}_{\mu}^{-}$ and using Proposition \ref{one-dim-projections}, we get
$$
\|{\bold P}_{\mu,-}^{+}(u,v)\|_{\bold X}^{2} 
 \leq 1/(1-(2/\alpha)/\mu)^2 \sum_{k\in N_{\mu}^{-}} \|(u_k\cdot e_k,v_k\cdot e_k) \|_{\bold X}^{2}.
$$
Now applying all the estimates for the components of \eqref{ortho-sum-proj-plus} we get
$$
\|{\bold P}_{\mu}^{+} (u,v)\|_{\bold X}^{2} 
\leq 1/(1-(2/\alpha)/\mu)^2 \sum_{k\in \mathbb{N}} \|(u_k\cdot e_k,v_k\cdot e_k)\|_{\bold X}^{2} = 
1/(1-(2/\alpha)/\mu)^2 \|(u, v)\|_{\bold X}^{2},
$$
which ends the proof.
\end{proof}
\begin{prop}\label{to-obtain-(D)-property}
For any $\mu>2/\alpha$ the following properties hold
\begin{enumerate}[{\em (i)}]
\item ${\bold X}_{\mu}^{+} \subset D({\bold A})$, ${\bold A} ({\bold X}_{\mu}^{+})\subset {\bold X}_{\mu}^{+}$;
\item ${\bold A} ({\bold X}_{\mu}^{-}\cap D({\bold A}))\subset {\bold X}_{\mu}^{-}$;
 \item $\| {\bold A} (u,v) \|_{\bold X} \leq \mu \cdot (1+\sqrt{2})    
 \| (u,v)\|_{\bold X}$ for any $(u,v)\in {\bold X}_{\mu}^{+}$;
\item $\left\langle {\bold A}(u,v), (u,v) \right\rangle_{\bold X} \leq - \mu \|(u,v)\|_{\bold X}^{2}$ for all $(u,v) \in {\bold X}_{\mu}^{-}\cap D({\bold A})$ and, in consequence,
$$
\| e^{t{\bold A}} (u,v)\|_{\bold X} \leq e^{-\mu t} \|(u,v)\|_{\bold X}
\ \ \mbox{ for any } (u,v)\in {\bold X}_{\mu}^{-}.
$$
\end{enumerate}
\end{prop}
\begin{proof} (i) Note that the space ${\bold X}_{\mu}^{+}$ is the closure of a sum of invariant spaces on which the operator ${\bold A}$ is bounded with the norms estimated by the same constant. Then the completeness of  $\bold X$ and the closedness of  ${\bold A}$ show that ${\bold X}_{\mu}^{+}\subset D({\bold A})$. The invariance is immediate.\\
\indent (ii) Take any $(u,v)\in {\bold X}_{\mu}^{-} \cap D({\bold A})$. Then
\begin{equation}\label{series-in-X-minus}
(u,v)= \sum_{k\in N_{\mu}^{-}} \alpha_k \cdot {\bold e}_{k}^{-}
\end{equation}
where $\alpha_k\in \mathbb{C}$, $k\in N_{\mu}^{-}$. In particular
$$
u = \sum_{k\in N_{\mu}^{-}} \alpha_k\cdot e_k \ \mbox{ and } \
v = \sum_{k\in N_{\mu}^{-}} \alpha_k \xi_{k}^{-}\cdot e_k.
$$
Since $v\in X^{1/2}$ and $u+\alpha\cdot v \in D(A)$ we have, in view of Corollary \ref{cor_properties_of_A-characterization} (i),
\begin{equation}\label{inv-convergence}
\sum_{k\in N_{\mu}^{-}} \lambda_k |\alpha_k\xi_{k}^{-}|^2  < +\infty \ \ \mbox{ and } \ \        \sum_{k\in N_{\mu}^{-}} \lambda_{k}^2 |1+\alpha\xi_{k}^{-}|^2|\alpha_k|^2 = 
\sum_{k\in N_{\mu}^{-}} |\xi_{k}^{-}|^4|\alpha_k|^2 < +\infty.
\end{equation}
Observe also that 
\begin{align*}
\|{\bold A}(\alpha_k \cdot{\bold e}_{k}^{-})\|_{\bold X}^{2} & 
= |\alpha_k \xi_{k}^{-}|^2 \|(e_k, \xi_{k}^{-}\cdot e_k)\|_{\bold X}^{2} =
(\lambda_k + |\xi_{k}^{-}|^2)|\xi_{k}^{-}|^2 |\alpha_k|^2, 
\end{align*}
which, in view of \eqref{inv-convergence}, implies the convergence of the series
$$
\sum_{k\in N_{\mu}^{-}}  {\bold A}(\alpha_k \cdot {\bold e}_{k}^{-}) 
\mbox{ in } {\bold X}.
$$
Hence, by use of the closedness of ${\bold A}$, we get 
\begin{equation}\label{A-value-on-X_muninus}
{\bold A}(u,v) = \sum_{k\in N_{\mu}^{-}} {\bold A}(\alpha_k \cdot{\bold e}_{k}^{-}) \in {\bold X}_{\mu}^{-}.
\end{equation}
This completes the proof of the invariance.\\
\indent (iii) Take any $(u,v)\in {\bold X}_{\mu}^{+}$. Using \eqref{spliting-mu-plus}, we have
\begin{align*}
{\bold P}_{\mu, \mathbb{R}}^{+} (u,v) 
& = \! \sum_{k\in N_{\mu, \mathbb{R}}^{+}}  (u_k \cdot e_k,v_k \cdot e_k),\\
{\bold P}_{\mu, \mathbb{C}}^{+} (u,v) 
& = \! \sum_{k\in N_{\mu, \mathbb{C}}^{+}}  (u_k \cdot e_k,v_k \cdot e_k),\\  
{\bold P}_{\mu,-}^{+} (u,v) 
& = \! \sum_{k\in N_{\mu}^{-}}  p_{k}^{+}(u_k,v_k) \cdot {\bold e}_{k}^{+},\\
{\bold P}^{0} (u,v) 
& = \! \sum_{k\in N^{0}}  (u_k \cdot e_k,v_k \cdot e_k).
\end{align*}
By the invariance and orthogonality  of the components
\begin{equation}\label{norm-estimate-eq}
\|{\bold A} {\bold P}_{\mu}^{+}(u,v) \|_{\bold X}^{2} = 
\|{\bold A} {\bold P}_{\mu, \mathbb{R}}^{+} (u,v) \|_{\bold X}^{2}
+ \|{\bold A} {\bold P}_{\mu, \mathbb{C}}^{+} (u,v) \|_{\bold X}^{2}
+ \|{\bold A} {\bold P}_{\mu,-}^{+} (u,v) \|_{\bold X}^{2}+\|{\bold A} {\bold P}^{0} (u,v) \|_{\bold X}^{2}.
\end{equation}
Clearly, by use of Proposition \ref{two-dim-operators} and the fact that $g$ is decreasing and 
$$  
(2/\alpha)^2/\lambda_k > 1-(1-(2/\alpha)/\mu)^2 \mbox{ for all } k\in N_{\mu, \mathbb{R}}^{+},
$$ 
we obtain
\begin{align*}
\|{\bold A} {\bold P}_{\mu, \mathbb{R}}^{+} (u,v) \|_{\bold X}^{2} & =  \sum_{k\in N_{\mu, \mathbb{R}}^{+}}  \left\| {\bold A} (u_k \cdot e_k,v_k \cdot e_k)  \right\|_{\bold X}^{2}\\
& \leq \left[ (2/\alpha) \cdot  g(1-(1-(2/\alpha)/\mu)^2) \right]^2 \cdot  \|{\bold P}_{\mu, \mathbb{R}}^{+} (u,v) \|_{\bold X}^{2} \\
& = \mu^2 \cdot 
\left[ \frac{1+\sqrt{1+\varrho(2-\varrho)}}{2-\varrho} \right]^2\cdot  \|{\bold P}_{\mu, \mathbb{R}}^{+} (u,v) \|_{\bold X}^{2}\\
& \leq \mu^2 \cdot (1+\sqrt{2})^2 \cdot  \|{\bold P}_{\mu, \mathbb{R}}^{+} (u,v) \|_{\bold X}^{2}
\end{align*}
where $(0,1)\ni\varrho:=(2/\alpha)/\mu$.
Next, again using Proposition \ref{two-dim-operators} and  the inequality
$$
(2/\alpha)^2 / \lambda_k > 1 \mbox{ for all }  k\in N_{\mu,\mathbb{C}}^{+}
$$
we have
$$
\|{\bold A} {\bold P}_{\mu, \mathbb{C}}^{+} (u,v) \|_{\bold X}^{2} =  \sum_{k\in N_{\mu, \mathbb{C}}^{+}}  \left\| {\bold A} (u_k \cdot e_k,v_k \cdot e_k) \right\|_{\bold X}^{2} \leq 
\left[ (2/\alpha) \cdot g(1)\right]^2 \|{\bold P}_{\mu, \mathbb{C}}^{+} (u,v) \|_{\bold X}^{2} .
$$
Furthermore, taking into considerations the inequality
$$
0\geq \xi_{k}^{+} \geq -\frac{2}{\alpha} \cdot \frac{1}{2-(2/\alpha)/\mu} \  \mbox{ for all } \  k\in N_{\mu}^{-},
$$
we get
\begin{align*}
\|{\bold A} {\bold P}_{\mu,-}^{+} (u,v) \|_{\bold X}^{2} & =  \sum_{k\in N_{\mu}^{-}} \left\|  {\bold A} (p_{k}^{+}(u_k,v_k)\cdot {\bold e}_{k}^{+}) \right\|_{\bold X}^{2} 
=  \sum_{k\in N_{\mu}^{-}} | \xi_{k}^{+} |^2 \left\| p_{k}^{+}(u_k,v_k) \cdot 
{\bold e}_{k}^{+}\right\|_{\bold X}^{2}\\
& \leq \left[ (2/\alpha)/(2-(2/\alpha)/\mu) \right]^2 \cdot\|{\bold P}_{\mu,-}^{+} (u,v) \|_{\bold X}^{2} \leq (2/\alpha)^2 \|{\bold P}_{\mu,-}^{+} (u,v) \|_{\bold X}^{2}. 
\end{align*}
Observe also that, by use of Proposition \ref{two-dim-operators}, one has
$$
\|{\bold A} {\bold P}^{0} (u,v)\|_{\bold X}^{2} =   \sum_{k\in N^{0}}  \left\| {\bold A} (u_k \cdot e_k,v_k \cdot e_k)  \right\|_{\bold X}^{2} \leq \left[ (2/\alpha)\cdot g(1)\right]^2\cdot  \|{\bold P}^{0} (u,v) \|_{\bold X}^{2}.
$$
Combining the above inequalities together with \eqref{norm-estimate-eq} we arrive at the desired assertion.\\
\indent (iv) Take any $(u,v)\in {\bold X}_{\mu}^{-}\cap D({\bold A})$. Clearly, there are 
$\alpha_k \in \mathbb{C}$, $k\in N_{\mu}^{-}$, such that \eqref{series-in-X-minus} holds. In view of \eqref{A-value-on-X_muninus} and the orthogonality of the components, we have
\begin{align*}
\left\langle {\bold A}(u,v), (u,v) \right\rangle_{\bold X} & = 
\sum_{k\in N_{\mu}^{-}} \left\langle {\bold A}(\alpha_k \cdot{\bold e}_{k}^{-}), \alpha_k\cdot {\bold e}_{k}^{-} \right\rangle_{\bold X} = \sum_{k\in N_{\mu}^{-}} \xi_k^{-} \|\alpha_k\cdot {\bold e}_{k}^{-}\|_{\bold X}^{2}\leq - \mu \|(u,v)\|_{\bold X}^{2}, 
\end{align*}
which ends the proof.
\end{proof}
\begin{wn}
The operator ${\bold A}$ has property $(UHBD)$.
\end{wn}
Thus, we can apply directly Corollary \ref{cor-1} with $\mu_0:=2/\alpha$ and $M:=1+\sqrt{2}$ to derive the estimate in Theorem \ref{main-res}.
\begin{uw}\label{remark-local-hyp}
We can refine our result by use of Remark \ref{remark-local-abs}.
Instead of the global Lipschitzianity of $f$, assume that, for any
$R, R'>0$, there exists $L=L_{R,R'}>0$ such that \eqref{E101} holds whenever  $\|u_1\|_{1/2}, \|u_2\|_{1/2}\leq R$ and $\|v_1\|, \|v_2\|\leq R'$. Then if $u:\mathbb{R}\to X^{1/2}$ is a $T$-periodic solution of \eqref{E25}, then $T \geq 1/L_{R,R'}\left(1+\sqrt{(1+1/\sqrt{2})(1+2/\alpha L_{R,R'})}\right)^2$.
\end{uw}

\section{Application to strongly damped beam equation}
As an illustration, let us consider the following damped beam equation
\begin{equation}\label{beam-eq}
\left\{ 
\begin{array}{ll}
u_{tt} + \alpha u_{txxxx} + \beta u_t + u_{xxxx} = h(x, u, u_t, u_x, u_{xx}), & \  x\in (0,l), \ t\geq 0,\\
u(0,t) = u(l,t)=0, & \ t>0\\ 
u_{xx}(0,t)=u_{xx}(l,t)=0, & \ t>0
\end{array} \right.
\end{equation}
where $l>0$, $\alpha, \beta>0$ and $h:[0,l]\times \mathbb{R}^4 \to \mathbb{R}$ is a continuous function such that 
$$
|h(x,z_1)-h(x, z_2)|\leq L |z_1-z_2| 
$$
for some fixed $L>0$ and all $z_1, z_2 \in\mathbb{R}^4$ and $x\in [0,l]$. Here $|\cdot |$ stands for either the Euclidean norm (in $\mathbb{R}^4$) or the absolute value (in $\mathbb{R}$). It is the so-called beam equation with strong damping coming from Ball's model for extensible beam coming from \cite{Ball}. The boundary conditions above can be replaced by $u(0,t)=u(l,t)=0$ and $u_x (0,t)=u_x (l,t)=0$, which corresponds to the way the beam's ends are fixed.\\
\indent If we define $A:D(A)\to L^2(0,l)$ with 
\begin{eqnarray*}
D(A):= \{ u\in L^2(0,l)\, \mid\, u\in W^{4,2}(0,l), u(0)=u(l)=0, \, u''(0)=u''(l)=0 \}\\
\mbox{or } \, D(A):=\{ u\in L^2 (0,l) \,\mid\, 
u\in W^{4,2} (0,l), u(0)=u(l)=0, \, u'(0)=u'(l)=0 \}
\end{eqnarray*}
(depending on the boundary conditions) and
$$
Au = u'''', \, u\in D(A).
$$
The space $X:=L^2(0,l)$ is equipped with the standard scalar product $\langle u,v\rangle_{L^2}:= \int_{0}^{l} u(s) v(s) \, d s$ and the norm $\|u\|_{L^2}:= \sqrt{\langle u,u\rangle_{L^2}}$, $u,v\in X$. The operator $A$ (in both versions) has compact resolvent and using its spectral representation one may show that $A^{1/2}u=-u''$, for $u\in X^{1/2} = D(A^{1/2})$
and $\|A^{1/4} u\|_{L^2} = \|u'\|_{L^2}$ for each $u\in X^{1/4}$. The mapping $f:X^{1/2}\times X \to X$ we define by
$$
[f(u,v)](x):= h(x,u(x), v (x), u_{x}(x), u_{xx}(x))-\beta v(x),  \, \mbox{ for a. e. } x\in [0,l].
$$ 
We denote eigenvalues of $A$ by $\lambda_k$, $k\in\mathbb{N}$. Observe that for all $k\in\mathbb{N}$ $\lambda_{k+1}\geq\lambda_k>0$ and $\lambda_k\to\infty$. One can directly verify that for any $u_1$, $u_2\in X^{1/2}$ and $v_1$, $v_2\in X$
$$
\|f(u_1,v_1)-f(u_2,v_2)\|_{X}\leq \widetilde{L}\cdot
\|(u_1,v_1)-(u_2,v_2)\|_{X^{1/2}\times X}.
$$
where  $\widetilde{L}:=\sqrt{2}L\Big(1+\max\{{\lambda_1}^{-1/4}+{\lambda_1}^{-1/2},\beta/L\}\Big)$.
Hence, we can directly apply Theorem \ref{main-res} to see that if \eqref{beam-eq} has a nontrivial $T$-periodic solution then 
$$
T \geq 1/\widetilde{L}\left(1+\sqrt{(1+1/\sqrt{2})(1+2/\alpha \widetilde{L})}\right)^2.
$$


\begin{thebibliography}{99}

\bibitem{Ball}  \textsc{Ball J.M.:} 
\textit{Stability theory for an extensible beam},
J. Differential Equations 14 (1973), 399–418.

\bibitem{Busenberg-et-al}\textsc{Busenberg S. N., Fisher D. C., Martelli M.:} \textit{Better bounds for periodic solutions of differential equations in Banach spaces}, Proc. Amer. Math. Soc. 98 (1986), 376–378.

\bibitem{Cwiszewski_Rybakowski} \textsc{\'Cwiszewski A., Rybakowski K. P.:} \textit{Singular dynamics of strongly damped beam equation}, J. Differential Equations 247 (2009), 3202–3233.

\bibitem{Dunford-Schwartz} \textsc{Dunford N., Schwartz J. T.:},
\textit{Linear Operators}, Parts I and II, Wiley-Interscience, New York 1966.

\bibitem{Fitzgibbon} \textsc{Fitzgibbon W.E.:}, \textit{Strongly damped quasilinear evolution equations}, J. Math. Anal. Appl. 79 (1981), 536-550.

\bibitem{Hale-book} \textsc{Hale J.:} \textit{Asymptotic behavior of dissipative system}, Mathematical Surveys and Monographs 25, American Mathematical Society 2007. 

 \bibitem{Henry}\textsc{Henry D.:} \textit{Geometric Theory of Semilinear Parabolic Equations}, Springer, Berlin 1981.

\bibitem{Massatt}\textsc{Massatt P.:} \textit{Limiting behavior for strongly damped nonlinear wave equations}, J. Differential Equations 48 (1983), 334–349.

\bibitem{Rob-Vid-Lop-2006} \textsc{Robinson J. C., Vidal-Lopez A.:} 
\textit{Minimal periods of semilinear evolution equations with Lipschitz nonlinearity}, J. Differential Equations 220 (2006), 396-406.

\bibitem{Rob-Vid-Lop}
\textsc{Robinson J. C., Vidal-Lopez A.:} \textit{Minimal periods of semilinear evolution equations with Lipschitz nonlinearity revisited}, J. Differential
Equations 254 (2013), 4279–4289.

\bibitem{Rudin}\textsc{Rudin W.:} \textit{Functional Analysis}, McGraw-Hill, 1991.

\bibitem{Yorke}\textsc{Yorke J. A.:} \textit{Periods of periodic solutions and the Lipschitz constant.} Proc. Amer. Math. Soc. 22 (1969), 509–512.

\end{thebibliography}
\end{document}